\documentclass[11pt,double,reqno]{amsart}

\usepackage{tikz}
\usepackage{times}
\usepackage{graphicx}
\usepackage{amssymb}
\usepackage{epsfig}
\usepackage{amsmath}
\usepackage{amsthm}
\usepackage{color}

\newtheorem{theo}{Theorem}[section]
\newtheorem{defin}[theo]{Definition}

\newtheorem{lemm}[theo]{Lemma}
\newtheorem{rem}[theo]{Remark}

\numberwithin{equation}{section}

\newcommand{\al}{\alpha}
\newcommand{\be}{\beta}

\newcommand{\Ga}{\Gamma}
\newcommand{\la}{\lambda}

\newcommand{\om}{\omega}
\newcommand{\Om}{\Omega}

\newcommand{\si}{\sigma}

\newcommand{\ep}{\epsilon }
\newcommand{\te}{\theta}
\newcommand{\De}{\Delta}
\newcommand{\de}{\delta}

\newcommand{\pa}{\partial}

\newcommand{\R}{{\mathbb R}^n}

\newcommand{\ri}{\rightarrow}
\newcommand{\Rn}{{\mathbb R}^{n-1}}

\newcommand{\na}{\nabla}

\begin{document}
\baselineskip=18pt

\title[]{Global in time solvability of the  Navier-Stokes equations  in    the half-space 
}

\
\author{Tongkeun Chang}
\address{Department of Mathematics, Yonsei University \\
Seoul, 136-701, South Korea}
\email{chang7357@yonsei.ac.kr}

\author{Bum Ja Jin}
\address{Department of Mathematics, Mokpo National University, Muan-gun 534-729,  South Korea }
\email{bumjajin@mokpo.ac.kr}

\thanks{Bum Ja Jin was supported by NRF-2016R1D1A1B03934133.}

\begin{abstract}
In this  paper, we study the  initial value problem of the Navier-Stokes equations in the half-space. Let a solenoidal  initial velocity be given in  the function space $ \dot{B}_{pq,0}^{\alpha-\frac{2}{2}}({\mathbb R}^n_+)$   for  $\alpha +1 = \frac{n}p + \frac2q$ and $0<\alpha<2$. We prove the global in time existence of weak solution $u\in L^q(0,\infty; \dot B^\alpha_{pq}({\mathbb R}^n_+))\cap L^{q_0}(0, \infty; L^{p_0}({\mathbb R}^n_+)) $   for some $ 1<p_0, q_0<\infty$ with    $\frac{n}{p_0} +\frac2{q_0} =1$, when the  given initial velocity   has small norm in function space $ \dot{B}_{p_0q_0,0}^{-\frac{2}{q_0}}({\mathbb R}^n_+)$. The solution is unique in the class $L^{q_0}(0, \infty; L^{p_0}({\mathbb R}^n_+))$. Pressure estimates are also given.
    \\




\end{abstract}

\maketitle

\section{\bf Introduction}
\setcounter{equation}{0}

In this  paper, we study the following nonstationary Navier--Stokes equations
\begin{align}\label{maineq2}
\begin{array}{l}\vspace{2mm}
u_t - \De u + \na p =-{\rm div}(u\otimes u), \qquad {\rm div} \, u =0 \mbox{ in }
 \R_+\times (0, \infty),\\
\hspace{30mm}u|_{t=0}= h, \qquad  u|_{x_n =0} = 0,
\end{array}
\end{align}
where
 $u=(u_1,\cdots, u_n)$ and $p$ are the unknown velocity and pressure, respectively,
   $     h=(h_1,\cdots, h_n)$ is the given initial data.

Since the nonstationary Navier--Stokes equations are invariant under the scaling
\begin{align*}
u_\la(x,t) = \la u(\la x, \la^2 t),\quad   p_\la (x,t) = \la^2 p(\la x, \la^2t), \quad
h_{\la} = \la h (\la x),
\end{align*}
it is important to
study \eqref{maineq2} in the so-called critical spaces, i.e., the function spaces
with norms invariant under the scaling $u(x,t) \ri \la u(\la x,\la^2 t)$.

There are a number of papers dealt with global well-posedness for \eqref{maineq2} in critical spaces in the half space.
Among them, R. Farwig, H. Sohr and W. Varnhorn  \cite{FSV} showed  that if $h \in \dot B^{-\frac{2}q}_{pq} (\R_+)$ with $\frac{n}p + \frac2q=1$ and $1<p,q<\infty$ has  sufficiently small norm, then \eqref{maineq2} has a unique solution  $u \in L^q(0, \infty; L^p(\R_+))$. Moreover, R. Farwig, Y. Giga and P. Hsu \cite{farwig-giga-hsu}, showed  if $h \in \dot B^{-1 +\frac{n}p}_{pq} (\R_+)$ with  $\frac{n}p + \frac2q =1-2\al, \, \frac12 \geq \al \geq 0$  and $1<p,q<\infty$, then \eqref{maineq2} has a unique solution satisfying
\begin{align*}
t^{2\al}u \in L^q (0, \infty; L^p(\R_+)).
\end{align*}
(Also see \cite{chang-jin2} and   \cite{CJ}).

R. Danchin and P. Zhang in \cite{DZ} have studied global solvability of inhomogeneous Navier-Stokes equations in the half space with bounded density, 
and showed that if the initial velocity in $ \dot B^{-1+\frac{n}p}_{pq} (\R_+) $ with $ \frac{n}3 <p < n$, $1<q<\infty$ and $q\geq \frac{2p}{3p-n}$ is small and initial density in $L^\infty(\R_+)$ is close enough to the homogeneous fluid, then \eqref{maineq2} has a unique solution  satisfying
\begin{align*}
t^{\al} \big( u_t, D_x^2 u, \na p \big) \in L^q(0, \infty; L^p(\R_+)), \ t^{\be}\nabla u\in L^{q_2}(0,\infty;L^{p_2}(\R_+)),\, t^{\gamma}u\in L^{q_3}(0,\infty;L^{p_3}(\R_+))
\end{align*}
for some $ \be, \gamma>0$, $1<p_2,p_3, q_2,q_3<\infty$ with $\al=\be+\gamma$, $\frac{1}{p}=\frac{1}{p_2}+\frac{1}{p_3}$ and $\frac{1}{q}=\frac{1}{q_2}+\frac{1}{q_3}$.

The limiting case $q=\infty$ has been studied by M. Cannone, F.  Planchon, and  M. Schonbek \cite{CPS} for $h \in L^3 ({\mathbb R}^3_+) (\subset B^{-1 +\frac3p}_{p \infty}({\mathbb R}^3_+))$,
by  H. Amann  \cite{amann} for $h \in b^{-1 +\frac{n}p}_{p,\infty} (\R_+)$, $ p > \frac{n}3,p \neq  n$, and by  M. Ri, P. Zhang and Z. Zhang \cite{RZZ} for  $h \in b^{0}_{n \infty} (\R_+)$,
where $b^s_{pq}(\Omega)$ denotes the completion of the generalized Sobolev space $H^s_p(\Omega)$ in $B^s_{pq}(\Omega)$.
In particular, in \cite{CPS}, the solution exists globally in time when $\|h\|_{B^{-1 +\frac3p}_{p \infty}({\mathbb R}^3_+)}$ is small enough.
See also   \cite{fernandes,FGH,giga1,kozono1,Ko-Ya,Ma,sol1} and the references therein for initial   value problem of Navier-Stokes equations in the half space.

Our study in this paper is motivated by the result in \cite{CPS} and  \cite{DZ}.
The following texts state our main results.
\begin{theo}
\label{thm-navier}
Let $0<\al<2$ and  $\al +1 =\frac{n}p +\frac2q$. Assume that  $h\in \dot {B}_{pq,0}^{\al-\frac{2}{q}}(\R_+)$ with $ \mbox{\rm div}\, h=0$. Then,    there is $\epsilon_*>0$  and    $ 1 <p_0, \, q_0<\infty$ with $\frac{n}{p_0}+\frac{2}{q_0}=1$ so that   if $\|h\|_{ \dot {B}_{p_0q_0,0}^{-\frac{2}{q_0}}(\R_+) }<\ep_*,$ then
 \eqref{maineq2} has a solution
 $u\in L^q(0,\infty;\dot B^\al_{pq} (\R_+))\cap    L^{q_0}(0, \infty; L^{p_0}(\R_+))$. 
The solution is unique in $L^{q_0}(0, \infty; L^{p_0}(\R_+)).$

\end{theo}
In particular, if $p =q$, then the velocity $u$ is contained in $\dot B^{\al, \frac{\al}2}_{pp} (\R_+ \times (0, \infty)) = L^p (0, \infty; \dot B^{\al}_{pp}) \cap L^p(\R_+: \dot B^{\frac{\al}2}_{pp}(0, \infty))$ (see \cite{CJ3}).

\begin{theo}\label{maintheopressure}
Let $(p,q,\al)$ and $h$ be conditions in Theorem \ref{thm-navier}. Then, there is  $(p_1,q_1,\be)$ with  $1<p_1< p,\ 1< q_1< q$,  and (1) $0 < \be \leq  \al $ if $0 < \al \leq 1$  and (2) $1 < \be \leq  \al$ if $ 1 < \al \leq 2$  so that
if $\al>\frac{1}{p}$, then  the corresponding pressure $p$ can be decomposed by  $p=P_0+\sum_{j=1}^{n-1}D_{x_j}P_j+D_tp_1$ for some  $p_1\in L^q(0,\infty;\dot B^{\al+1}_{pq} (\R_+))$,  $P_j\in L^q(0,\infty;\dot B^{\al}_{pq} (\R_+))$, $P_0\in L^{q_1}(0,\infty;\dot B^\be_{p_1 q_1} (\R_+))$. 

\end{theo}
The explanation of function spaces and notations is placed in Section \ref{notation}.

\begin{tikzpicture}[scale=3]

\label{figure}

 \draw (-2,-0.6) -- (2,-0.6);

\draw (-1.0, 1.0) -- (-1.0,-1.0);

\draw(-0.3, 0.60)-- (0.9,-0.6);
\draw(-1.0, 0.6)-- (-0.3,0.6);

\draw (1.3,-0.7 ) node {$\frac{n}p$};

\draw (0.9,-0.7 ) node {$\frac{n}p =3 $};

\draw (-0.3,-0.7 ) node {$\frac{n}p =1 $};

\draw(-1.0, 0.03) -- ( -0.3, -0.6);

\draw(-0.3, -0.6) -- (-0.3, 0.6);

\draw(-1.2, 0.1) node {$\frac{2}q =1$};

\draw(-1.2, 0.6) node {$\frac{2}q =2$};

\draw(-1.1, 1.0) node {$\frac{2}q $};

\draw(-0.8, -0.4) node { {\Large I} };

\draw(0.1, -0.25) node { {\Large II} };

\draw(-0.65, 0.08) node { {\Large III} };



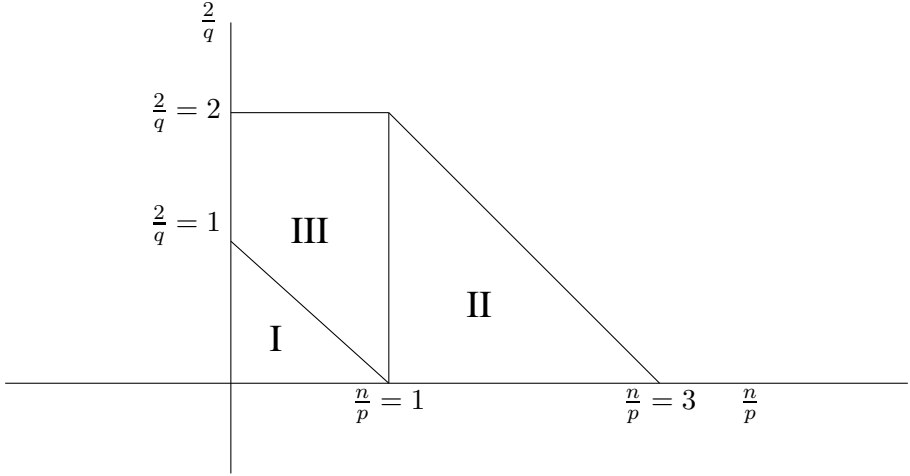
\draw(0.0, -1.2) node {{\bf  Figure: Region of  $(\frac{n}p, \frac2q)$} };

\end{tikzpicture}

\begin{rem}

\begin{itemize}
\item[(1)]
The authors in the papers \cite{farwig-giga-hsu} and  \cite{FSV} studied \eqref{maineq2} when the initial data $h$ is in space $\dot B^{-1 +\frac{n}p}_{pq,0}(\R_+)$ with $(\frac{n}p, \frac2q)$ in I and the authors in the paper \cite{DZ} studied  \eqref{maineq2} when  $(\frac{n}p, \frac2q)$ is in   II. In this paper, we  study  \eqref{maineq2} when $(\frac{n}p, \frac2q)$ is in  II $\cup$ III.

\item[(2)]
Note that if $(\frac{n}p, \frac2q)$ is  in I and III, then $-1 +\frac{n}p < 0$.

\end{itemize}
\end{rem}

For the proof of Theorem \ref{thm-navier}, it is necessary to study the following initial value problem of the Stokes equations in $\R_+\times (0,\infty)$:
\begin{align}\label{maineq-stokes}
\begin{array}{l}\vspace{2mm}
u_t - \De u + \na p =f, \qquad {\rm div} \, u =0 \mbox{ in }
 \R_+\times (0,\infty),\\
\hspace{30mm}u|_{t=0}= h, \qquad  u|_{x_n =0} = 0,
\end{array}
\end{align}
where $f=\mbox{div}{\mathcal F}$.

In \cite{GGS}, M. Giga, Y. Giga and H. Sohr showed that if $f \in L^q(0, T; \hat D(A^{-\al}_p(\Om))$ and $h=0$ then the solution $u$ of Stokes equations  \eqref{maineq-stokes} satisfies that for $0 < \al <1$,
\begin{align*}
\int_0^T \Big( \| (\frac{d}{dt})^{1-\al}  u(t) \|^q_{L^p (\Om)} + \| A^{1-\al}_p u(t)\|^q_{L^p (\Om)}  \Big)dt \leq c(p,q, \Om,\al) \int_0^T \|A_p^{-\al}f (t) \|^q_{L^p (\Om)} dt,
\end{align*}
where   $A_p $ is Stokes operator and $\Om$ is bounded domain, exterior domain or half space. In particular,  if $f ={\rm div}\, {\mathcal F}$ with ${\mathcal F} \in L^q(0, T; L^p_\si (\Om))$ then
\begin{align*}
\int_0^T \Big(\| (\frac{d}{dt})^{\frac12}  u(t) \|^q_{L^p (\Om)} + \| \na u(t)\|^q_{L^p (\Om)} \Big) dt \leq c(p,q, \Om) \int_0^T \|{\mathcal F}(t) \|^q_{L^p (\Om)} dt.
\end{align*}
Estimates for the pressure were, however, not given in \cite{GGS}.

H. Koch and V. A. Solonnikov~\cite{KS} showed  the unique local in time existence of solution $u \in L^q(0,T; L^q({\mathbb R}^3_+))$  of \eqref{maineq-stokes}
when $f=\mbox{div}{\mathcal F}$, ${\mathcal F}\in L^q(0,T; L^q({\mathbb R}^3_+))$ and  $h=0$. They also showed that the corresponding pressure $p$ is decomposed by $p = p_1 + \frac{\pa P}{\pa t}$, where $p_1$ and $P$ satisfy $\| p_1\|_{L^q(0, T; L^q({\mathbb R}^3_+))} + \| P\|_{L^q(0, T; W^2_q({\mathbb R}^3_+))} \leq c \| {\mathcal F}\|_{L^q(0, T; L^q ({\mathbb R}^3_+))}$. 
See also  \cite{giga1,KS1,kozono1,sol1,Sol-2} and the references therein.

The following theorem states our result on the unique solvability of the Stokes equations \eqref{maineq-stokes}.
\begin{theo}\label{thm-stokes}
Let $1 < p,q < \infty$ and $0\leq \al\leq 2$.
Let  ${\mathcal F} \in L^{q_1} (0, \infty, \dot B^{\be}_{p_1 q}(\R_+))$ for some $(p_1,q_1,\be)$ satisfying  $1 < p_1 \leq p ,  \, 1 < q_1 \leq q$, $0\leq \beta\leq \al\leq\be+1\leq 2$ and  $ 0=\al -\be -1 +n (\frac{1}{p_1}-\frac{1}{p})+\frac2{q_1}-\frac2{q}$. Moreover assume that  ${\mathcal F}|_{x_n =0} =0$.  Then
  there is a solution  $u$ of \eqref{maineq-stokes}   with
\begin{align}\label{0411-1-1}
\|   u \|_{L^q (0, \infty;  \dot B^{\al}_{pq}({\mathbb R}^n_+))}
   \leq   c\big(  \|    h\|_{\dot B^{\al -\frac2q}_{pq,0} ({\mathbb R}^n_+)} + \| {\mathcal F}\|_{L^{q_1} (0, \infty, \dot B^{\be }_{p_1q}(\R_+))}\big).
\end{align}

The corresponding pressure $p$ can be decomposed by  $p=D_tp_1+\sum_{j=1}^{n-1}D_{x_j}P_j+P_0$,  $p_1\in L^q(0,\infty;\dot B^{\al+1}_{pq} (\R_+))$, $P_j\in L^q(0,\infty;\dot B^{\al}_{pq} (\R_+))$ and  $P_0\in L^{q_1}(0,\infty;\dot B^\be_{p_1q} (\R_+))$ with
\begin{align}\label{main0910}
\notag \|p_1\|_{ L^q(0,\infty;\dot B^{\al+1}_{pq} (\R_+))}+\sum_{j=1}^{n-1}\|P_j\|_{L^q(0,\infty;\dot B^{\al}_{pq} (\R_+))}+\|P_0\|_{L^{q_1}(0,\infty;\dot B^\be_{p_1q} (\R_+))}\\
\leq c\big(  \|    h
\|_{\dot B^{\al -\frac2q}_{pq,0} ({\mathbb R}^n_+)} + \| {\mathcal F}\|_{L^{q_1} (0, \infty, \dot B^{\be }_{p_1q}(\R_+))}\big).
\end{align}

\end{theo}

We organize this paper as follows.
In Section \ref{notation}, we introduce  the function spaces, definition of the weak solutions of Stokes equations and Navier-Stokes equations.  In Section \ref{preliminary},   the various  estimates of operators related with  Newtonian  kernel and Gaussian kernel are given.  
In Section \ref{theoremthm-stokes}, we  complete the proof of Theorem \ref{thm-stokes}. In Section \ref{nonlinear}, we give the proof of Theorem \ref{thm-navier} and Theorem \ref{maintheopressure} applying the estimates in  Theorem \ref{thm-stokes} to the  approximate solutions.

\section{Notations, Function spaces and  Definitions of weak solutions}

\label{notation}
We denote by  $x'$ and $x=(x',x_n)$ the points of the spaces $\Rn$ and $\R$, respectively.
The multiple derivatives are denoted by $ D^{k}_x D^{m}_t = \frac{\pa^{|k|}}{\pa x^{k}} \frac{\pa^{m} }{\pa t}$ for multi-index
$ k$ and nonnegative integers $ m$.
Throughout this paper we denote by $c$ various generic constants. 

%

For $s\in {\mathbb R}$ and $1\leq p,q\leq \infty$,
 we denote   $\dot H^s_{p}(\R)$ and $\dot{B}^s_{pq}(\R)$  the generalized homogeneous Sobolev spaces(space of Bessel potentials) and the homogeneous Besov spaces in $\R$, respectively (see \cite{BL,Tr} for the definition of function spaces).  
 Denote by $\dot H^s_p(\R_+)$ and $ \dot B^s_{pq}(\R_+)$  the restrictions of $\dot H^s_p (\R)$ and $  \dot B^s_{pq} (\R)$, respectively,    with norms
 \begin{align*}
\| f \|_{\dot H^s_p(\R_+) }  = \inf \{ \| F\|_{ \dot H^s_p(\R)}\, | \, F|_{\R_+} =f, \,\,  F \in  \dot H^s_p(\R) \},\\
\| f \|_{\dot B^s_{pq}(\R_+) }  = \inf \{ \| F\|_{ \dot B^s_{pq}(\R)}\, | \, F|_{\R_+} =f, \,\, F \in  \dot B^s_{pq}(\R) \}.
 \end{align*}
For    a non-negative integer $k$, $\dot H^k_p(\R_+)
 = \{ f \, | \,   \sum_{|l| =k} \| D^l f\|_{L^p (\R_+)} < \infty \}.$
In particular, $\dot H^0_p(\R_+) = L^p (\R_+)$. 


For $s\in {\mathbb R}$, we denote by  $\dot{B}^s_{pq}(\R_+), 1\leq p,q\leq \infty$  the  usual homogeneous Besov space in $\R_+$   and
 denote by
 \begin{align*}
 \dot{B}^s_{pq,0}(\R_+ ) & = \{ f \in  \dot{B}^s_{pq}(\R_+) \, | \,  \tilde f \in \dot B^s_{pq} ({\mathbb R}^n) \},
 \end{align*}
where $ \tilde f$ is zero extension of $f$ over $\R$. Note that $\| \tilde f \|_{\dot B^s_{pq} (\R)} \leq c \| f\|_{\dot B^s_{pq,0} (\R_+)}$.

Note that for $s \geq 0$,   $\dot B^{-s}_{pq, 0} (\R_+)$ is the dual space of  $\dot B^{s}_{p'q'} (\R_+)$, that is,  $\dot B^{-s}_{pq, 0} (\R_+) = (\dot B^s_{p' q'}(\R_+))^*$, where $\frac1p + \frac1{p'} =1$ and $ \frac1q +\frac1{q'} =1$.  Note that for $f \in  \dot B^s_{pq, 0} (\R_+)$, the zero extension $\tilde f$ of $f$ is function in $ \dot B^s_{pq } (\R)$ with $\| \tilde f\|_{\dot B^s_{pq} (\R)} \leq c \| f\|_{\dot B^s_{pq}(\R_+)}$.

For the Banach space $X$, we denote  by $L^q(0, \infty ;X), 1\leq q\leq \infty$  the usual Bochner space with norm
\begin{align*}
\| f\|_{L^q(0, \infty; X)} : = \big( \int_0^\infty \| f(t) \|_X^q dt \big)^\frac1q.
\end{align*}
%

For  $1<q<\infty$ and   $0< \theta<1$, we    denote by   $(X,Y)_{\theta,q}$ and $[X,Y]_\theta$ the real interpolation and complex interpolation, respectively, of the Banach space $X$ and $Y$. In particular, for $ 0< \te < 1 $, $  \al, \al_1, \al_2 \in {\mathbb R} $ and  $1 < p_1, p_2, q_1, q_2,  p, q, r < \infty$,
\begin{align}
\label{interpolation1}
[\dot H^{\al_1}_{p_1}(\R_+), \dot H^{\al_2}_{p_2}(\R_+) ]_\te = \dot H^\al_p(\R_+), \qquad
(\dot H^{\al_1}_{p}(\R_+), \dot H^{\al_2}_{p} (\R_+))_{\te, r} = \dot B^\al_{pr}(\R_+),\\
\label{interpolation1-2}
(\dot B^{\al_1}_{pq}(\R_+), \dot H^{\al_2}_{pq} (\R_+))_{\te, r} = \dot B^\al_{pr}(\R_+),\\
\label{interpolation2}
[ L^{q_1}(0, \infty; X), L^{q_2} (0, \infty; Y)]_\te = L^q (0, \infty; [X,Y]_\te),\\
\label{realinterpolation2}
( L^{q_1}(0, \infty; X), L^{q_2} (0, \infty; Y))_{\te,q} = L^q (0, \infty; (X,Y)_{\te,q}),
\end{align}
when $\al = \te \al_1 + (1 -\te) \al_2$,  $\frac1p = \frac{\te}{p_1} + \frac{1 -\te}{p_2}$ and $\frac1q = \frac{\te}{q_1} + \frac{1 -\te}{q_2}$.
See Theorem 6.4.5, Theorem 5.1.2 and Theorem 5.6.2 in \cite{BL}.

\begin{defin}[Weak solution of  the Stokes equations]
\label{stokesdefinition}
Let  $1<p,q<\infty$ and $0\leq \al\leq 2$.
Let $h,   {\mathcal F}$ satisfy the same hypotheses as in Theorem \ref{thm-stokes}.
A vector field $u\in L^q(0,\infty; \dot H^\al_p(\R_+))$  is called a weak solution of the Stokes equations \eqref{maineq-stokes} if the following conditions are satisfied:
 \begin{align*}
-\int^\infty_0\int_{\R_+}u\cdot \Delta \Phi dxdt&=\int^\infty_0\int_{\R_+} \big( u\cdot \Phi_t-{\mathcal F}:\nabla \Phi \big) dxdt
-\int_{\R_+} h(x) \cdot \Phi(x,0) dx
\end{align*}
for each $\Phi\in C^\infty_0(\overline{\R_+}\times [0,\infty))$ with ${\rm div} _x\Phi=0$, $\Phi\big|_{x_n=0}=0$. In addition, for each $\Psi\in C^1_c(\overline{\R_+})$
\begin{equation}\label{Stokes-bvp-2200}
\int_{\R_+} u(x,t) \cdot \na \Psi(x) dx =0  \quad  \mbox{ for all}
\quad 0 < t< \infty.
\end{equation}
\end{defin}

\begin{defin}[Weak solution to the Navier-Stokes equations] Let $1<p,q<\infty$ and $0 \leq \al \leq 2$ with $\al +1= \frac{n}{p}  +\frac{2}{q}$. 
Let $h$ satisfy the same hypothesis as in Theorem \ref{thm-navier}.
A vector field $u\in L^q (0,\infty; \dot H^\al_p(\R_+))$ is called a weak solution of the Navier-Stokes equations \eqref{maineq2} if the following variational formulations are satisfied:\\
 \begin{align}\label{weaksolution-NS}
-\int^\infty_0\int_{\R_+}u\cdot \Delta \Phi dxdt&=\int^\infty_0\int_{\R_+} \big(u\cdot \Phi_t+(u\otimes u):\nabla \Phi \big)  dxdt
-\int_{\R_+} h(x) \cdot \Phi(x,0) dx
\end{align}
for each $\Phi\in C^\infty_0(\overline{\R_+}\times [0,\infty))$ with $\mbox{\rm div} _x\Phi=0$, $\Phi\big|_{x_n=0}=0$. In addition, for each $\Psi\in C^1_c(\overline{\R_+})$, $u$ satisfies  \eqref{Stokes-bvp-2200}.
\end{defin}

\begin{rem}
If  $ 0< \al < \frac2q$, then the term
$\int_{\R_+}h(x) \cdot  \Phi (x, 0)dx$
 should be replaced by
$<h,\Phi(\cdot, 0)>$, where $<\cdot,\cdot>$ 
is the duality pairing between  $\dot B^{\al-\frac2q}_{pq,0}(\R_+ )$ and $\dot B^{-\al+\frac2q}_{p'q'}(\R_+ )$.

\end{rem}

\section{\bf Preliminary Estimates.}

\label{preliminary}
\setcounter{equation}{0}

\subsection{Trace theorem}

The following lemma is well known trace theorem  (see the proof of   Theorem 6.6.1  in \cite{BL} for (1) and Theorem 3.2.2  in \cite{galdi} for (2)).
\begin{lemm}\label{trace}
Let $1< p, q< \infty$.
\begin{itemize}
\item[(1)]
If   $ f \in \dot B^{\al}_{pq}(\R_+)$  for $\al > \frac1p$, then $f |_{x_n =0} \in \dot B^{\al -\frac1p}_{pq} (\Rn)$ with
\begin{align*}
 \| f |_{x_n =0} \| _{\dot B_{pp}^{\al -\frac1p}(\Rn)} \leq c \| f\|_{\dot H^\al_{p} (\R_+)}, \,\,\,  \| f |_{x_n =0} \| _{\dot B_{pq}^{\al -\frac1p}(\Rn)} \leq c \| f\|_{\dot B^\al_{pq} (\R_+)}.
\end{align*}
\item[(2)]
If $f \in L^p (\R_+)$ and  ${\rm div }\, f =0$ in $\R_+$, then   $f_n |_{x_n =0} \in \dot B^{-\frac1p}_{pp} (\Rn)$    with
\begin{align*}
\| f_n|_{x_n =0} \|_{\dot B^{-\frac1p}_{pp}(\Rn)} \leq c\| f\|_{L^p (\R_+)}.
\end{align*}
\end{itemize}

\end{lemm}

%
%
%
%

\subsection{Newtonial potential}

The fundamental solution of the Laplace equation in $\R$ is  denoted by
\[
   N(x) = \left\{\begin{array}{ll}
 \vspace{2mm}
  \frac{1}{\om_n (2-n)|x|^{n-2}}&\mbox{if }n\geq 3,\\
 \frac{1}{2\pi}\ln |x|&\mbox{if }n=2,\end{array}\right.
 \]
$\omega_n$ is the surface area of the unit sphere in $\R$.

We define $Nf$ by
\[N f(x)
=
\int_{\Rn} N(x'-y',x_n)f(y')dy'.\]

Observe that $D_{x_n}Nf$ is Poisson operator of Laplace equation in $\R_+$ and $D_{x_i}Nf=D_{x_n}NR_i'f$  for $i\neq n$, where   $R'=(R_1,\cdots, R_{n-1})$ is the $n-1$ dimensional Riesz operator.
Poisson operator is bounded from $\dot{B}^{\al-\frac{1}{p}}_{pp}(\Rn)$ to $\dot H^\al_{p}(\R_+), \al \geq 0$ and $R'$ is bounded from $\dot B^s_{pq}(\Rn)$ to $\dot B^s_{pq}(\Rn)$, $s\in {\mathbb R}$ (See \cite{St}).
Hence  the following estimates hold.
\begin{lemm}
\label{poisson1}
Let $\al\geq 0$, $1<p<\infty$ and $1 \leq q \leq \infty$. Then
\begin{align}\label{Poisson}
 \| \nabla_x Nf\|_{\dot H^\al_{p}(\R_+)}\leq c\|f\|_{\dot B^{\al-\frac{1}{p}}_{p}(\Rn)}, \qquad   \| \nabla_x Nf\|_{\dot B^\al_{pq}(\R_+)}\leq c\|f\|_{\dot B^{\al-\frac{1}{p}}_{pq}(\Rn)}.
 \end{align}
\end{lemm}

%

According to Calderon-Zygmund inequality
\begin{align*}
\|  \int_{{\mathbb R}^n} \nabla_x^2N(\cdot-y)    f(y) dy \|_{ L^p(\R_+)} \leq c \| f\|_{ L^p(\R_+)}\quad \mbox{ for } \quad 1 < p < \infty.
\end{align*}
Using Lemma \ref{poisson1}, (1) of  Lemma \ref{trace} and Calderon-Zygmund inequality the following estimates also hold.
\begin{lemm}\label{lemma0929-22}
For $\al \geq  0$ and $1 < p < \infty$.
\begin{align*}
\| \nabla_x^2 \int_{{\mathbb R}^n_+} N(\cdot-y)    f(y) dy \|_{ \dot H^\al_{p} (\R_+)} \leq c \| f\|_{ \dot H^\al_{p}  (\R_+)}.
\end{align*}
\end{lemm}

\subsection{Gaussian kernel}

The fundamental solution of the heat equation in $\R$ is denoted by
\[
 \Gamma(x,t)=\left\{\begin{array}{ll} \vspace{2mm}
 \frac{1}{ (2\pi t)^{\frac{n}{2}}}e^{-\frac{|x|^2}{4t}}&\mbox{if }t>0,\\
 0& \mbox{if }t\leq 0.
 \end{array}\right.
 \]


Define   $\Gamma_t*f(x)=\int_{\R}\Gamma(x-y,t)f(y) dy$.
The  norm of the homogeneous Besov space $\dot{B}^s_{pq}(\R)$  has the following equivalence:
\begin{align*}
 \|f\|_{\dot B^s_{pq}(\R)}\equiv \Big(\int^\infty_0\big\| t^{k-\frac{s}{2}}D^k_t\Gamma_t*f\big\|_{L^p(\R)}^q\frac{dt}{t}   \Big)^{\frac{1}{q}}
 \end{align*} for nonnegative integer $k>\frac{s}{2}$  (See \cite{Tr} for the reference).
By interpolation theorem, we have the following estimates.
\begin{lemm}
\label{gauss.equiv}

\begin{align*}
\|\Ga_t *f\|_{L^q(0,\infty;\dot{H}^\al_p(\R))}\leq c \|f\|_{\dot B^{\al-\frac{2}{q}}_{pq}(\R)}, \qquad \|\Ga_t *f\|_{L^q(0,\infty;\dot{B}^\al_{pq}(\R))}\leq c \|f\|_{\dot B^{\al-\frac{2}{q}}_{pq}(\R)}.
 \end{align*}
\end{lemm}

\subsection{ H$\ddot{\rm o}$lder type inequality}

The following H$\ddot{\rm o}$lder type inequality  is well-known result (see Lemma 2.2 in \cite{chae}). For $\be >0$, $\frac1{r_i} + \frac1{s_i} = \frac1p$, $i=1,2$,
\begin{align}\label{bilinear1}
\| f_1f_2\|_{\dot B^{\be}_{pq} (\R  ) }  \leq
        c \big( \| f_1\|_{ \dot  B^{\be }_{s_1 q} (\R  ) } \| f_2\|_{  L^{r_1} (\R  )}   + \| f_1\|_{ L^{s_2 }(\R  )} \|f_2\|_{ \dot  B^{\be }_{r_2 q}(\R   )}  \big).
\end{align}

Let $g$ be a function defined in $\R_+ $.
Let  $\tilde g$  be the Adam's extension of $g$ over $ \R  $ (see Theorem 5.19 in \cite{AF}). Then, for $1 \leq p, q \leq \infty$ and $0 < \be$,  we have
\begin{align}
\label{extension2}
\| \tilde g\|_{ \dot  B^{\be }_{p q} (\R  ) }  \leq c\| g\|_{ \dot  B^{\be }_{p q} (\R_+  ) }, \quad  \| \tilde g\|_{  L^p (\R  )}  \leq  c\| g\|_{  L^p (\R_+  )}.
\end{align}
Using  \eqref{bilinear1} and  \eqref{extension2}, the following estimates hold.
\begin{lemm}\label{0510prop}
Let $0 < \be  $ and $1 \leq p, q \leq \infty$. Then, for   $\frac1{r_i} + \frac1{s_i} = \frac1p$, $i=1,2$,
\begin{align*}
\| f_1f_2\|_{\dot B^{\be}_{pq} (\R_+  ) }  \leq
        c \big( \| f_1\|_{ \dot  B^{\be }_{s_1 q} (\R_+  ) } \| f_2\|_{  L^{r_1} (\R_+  )}   + \| f_1\|_{ L^{s_2 }(\R_+  )} \|f_2\|_{ \dot  B^{\be }_{r_2 q}(\R_+   )}  \big).
\end{align*}
\end{lemm}

%
%

%
\subsection{Helmholtz projection} \label{projection}

The Helmholtz projection ${\mathbb P}$ in the half-space $\R_+$ is given  by
\begin{align}\label{Hprojection}
{\mathbb P} f = f -\na {\mathbb Q}f = f - \na {\mathbb Q}_1f - \na {\mathbb Q}_2 f,
\end{align}
where ${\mathbb Q}_1 f$ and ${\mathbb Q}_2f $ satisfy the following equations;
\begin{align*}
\De {\mathbb Q}_1 f ={\rm div}\, f,\qquad
{\mathbb Q}_1 f|_{x_n =0} =0
\end{align*}
and
\begin{align*}
\De{\mathbb Q}_2f =0,\qquad
D_{x_n} {\mathbb Q}_2f|_{x_n =0} = \big(f_n -D_{x_n}{\mathbb Q}_1f\big)|_{x_n =0}.
\end{align*}
Note that ${\mathbb Q}_1f$ and ${\mathbb Q}_2f$ are represented by
\begin{align}\label{0427-1}
{\mathbb Q}_1f(x)& =- \int_{\R_+}  D_{y_i} \big(N(x- y) - N (x - y^* ) \big) f_i(y) dy,\\
\label{0427-2} {\mathbb Q}_2f (x) & = \int_{\Rn}   N(x'-y',x_n)  \big(f_n(y', 0)  - D_{y_n} {\mathbb Q}_1f (y', 0) \big) dy'.
\end{align}
Note that $\mbox{div}{\mathbb P}f=0$ and $({\mathbb P}f)_n|_{x_n =0} =0$.

\begin{lemm}\label{hemoz}
Let  $f=\mbox{div}{\mathcal F}$ with ${\mathcal F}|_{x_n =0} =0$. For $\al \geq 0$,
\begin{align*}
\| {\mathbb Q} f\|_{\dot H^\al_{p} (\R_+)} \leq c \|{\mathcal F}\|_{\dot H^\al_p (\R_+)}, \qquad \| {\mathbb Q} f\|_{\dot B^\al_{pq} (\R_+)} \leq c \|{\mathcal F}\|_{\dot B^\al_{pq} (\R_+)}.
\end{align*}

\end{lemm}
\begin{proof}
The proof of Lemma \ref{hemoz} is given in Appendix \ref{appendix.hemoz}.

\end{proof}

\begin{lemm}\label{0929-1}
Let $1 < p_1 \leq p < \infty, \, 1 < q_1 \leq q < \infty$. Let (1) $0 < \be \leq  \al $ if $0 < \al \leq 1$  and (2) $1 < \be \leq  \al$ if $ 1 < \al \leq 2$ such that $ 0=\al -\be -1 +n (\frac{1}{p_1}-\frac{1}{p})+\frac2{q_1}-\frac2{q}$.  Let $f=\mbox{div}\mathcal F$,
where  ${\mathcal F} \in L^{q_1} (0, \infty, \dot B^{\be}_{p_1q}(\R_+)) $ with ${\mathcal F} |_{x_n =0} =0$.  Then,
\begin{align*}
\|  \Ga* {\mathbb P}f\|_{L^q (0, \infty;  \dot B^\al_{pq}(\R_+))}, \,\, \|  \Ga^* * {\mathbb P}f\|_{L^q(0, \infty;  \dot B^\al_{pq}(\R_+))} \leq c\| {\mathcal F}\|_{L^{q_1} (0,\infty;  \dot B^{\be }_{p_1q}(\R_+))}.
\end{align*}
Here $\Gamma *f:=\int^t_0\int_{\R_+}\Gamma(x-y,t-s)f(y,s)dyds$ and $\Gamma^* *f:=\int^t_0\int_{\R_+}\Gamma(x'-y',x_n+y_n,t-s)f(y,s)dyds.$

\end{lemm}
\begin{proof}
The proof  of Lemma \ref{0929-1} is given in Appendix \ref{appendix0292-1}.
\end{proof}

\section{Proof of  Theorem \ref{thm-stokes}}
\label{theoremthm-stokes}
\setcounter{equation}{0}

First, we decompose the Stokes equation \eqref{maineq-stokes} as the following two equations:
\begin{align}
\label{stokes.zero}
\notag v_t - \De v+\nabla \pi  =0, \qquad {\rm div} \, v =0 \qquad \mbox{ in
}\,\,\R_+ \times (0,\infty),\\
v|_{t =0} = h \quad \mbox{ and }\quad v|_{x_n =0} =0,
\end{align}
and
\begin{align}\label{maineq-stokesh=0}
\begin{array}{l}\vspace{2mm}
V_t - \De V + \na \Pi =\mbox{div}{\mathcal F}, \qquad \mbox{div } V =0, \mbox{ in }
 \R_+\times (0,\infty),\\
\hspace{30mm}V|_{t=0}= 0, \qquad  V|_{x_n =0} = 0.
\end{array}
\end{align}

Let $u = V + v$ and $p =\pi+\Pi$. Then, $(u,p)$ is solution of  \eqref{maineq-stokes}.

\subsection{Estimate of $(v,\pi)$}

Define $(v,\pi)$ by \begin{equation}\label{expression-v-zero}
v_i (x,t) = \int_{{\mathbb R}^n_+} G_{ij}(x,y, t)
h_j(y) dy,
\end{equation}
\begin{equation}\label{expression-p-zero}
\pi(x,t) = \int_{{\mathbb R}^n_+} P(x,y, t)
\cdot  h (y)dy,
\end{equation}
where $G$ and $P$ are defined by
\begin{align}\label{formulas-v}
\notag G_{ij} &= \de_{ij} (\Ga(x-y, t) - \Ga(x-y^*,t ))\\
& \qquad  + 4(1 -\de_{jn})
\frac{\pa}{\pa x_j} \int_0^{x_n} \int_{{\mathbb R}^{n-1}}
            \frac{\pa N(x-z)}{\pa x_i} \Ga(z -y^* , t) dz,
\end{align}
\begin{align}\label{formulas-p}
\notag P_j(x,y,t) & =4 (1 - \de_{jn}) \frac{\pa }{\pa x_j}\Big[ \int_{{\mathbb
R}^{n-1}} \frac{\pa N(x' - z', x_n)}{\pa x_n} \Ga(z' -y', y_n,t) dz'\\
& \quad +\int_{{\mathbb R}^{n-1}} N(x' -z',x_n) \frac{\pa \Ga(z'-y', y_n,
t)}{\pa y_n} dz'\Big].
\end{align}
Then $(v,\pi)$ satisfies   \eqref{stokes.zero} (see  \cite{So}).

From Section 4 in \cite{CK}, we have the following estimate;
\begin{align*}
\| v\|_{L^q (0,\infty;    \dot H^{\al_i }_{p} (\R_+)) } & \leq c \|    h
\|_{\dot B^{\al_i -\frac2q}_{p q,0} ({\mathbb R}^n_+)}, \qquad  0\leq \al_i\leq 2, \quad i = 1,2.
\end{align*}
Using the property of real interpolation (see \eqref{interpolation1}, \eqref{interpolation1-2} and \eqref{realinterpolation2}), we have
\begin{align}\label{CK120-april10}
\| v\|_{L^q (0,\infty;    \dot B^{\al }_{pq} (\R_+)) } & \leq c \|    h
\|_{\dot B^{\al -\frac2q}_{p q,0} ({\mathbb R}^n_+)}, \qquad  0\leq \al\leq 2.
\end{align}

For $\al>\frac{1}{p}$,  the  following estimates for  $\pi$  also hold.
\begin{lemm}\label{lemma-1-2-zero}
Let $ 1 < p, q<\infty$ and  $\frac{1}{p}< \al <2$. Then  $\pi$ can be decomposed in the form   $\pi=\sum_{j=1}^{n-1}D_{y_j}\pi_{0j}+D_t\pi_{00}$ with
\begin{align*}
  \|  \pi_{0j}\|_{L^q (0,\infty; \dot H^{\al}_p (\R_+)) }+\|  \pi_{00}\|_{L^q (0,\infty; \dot H^{\al+1}_p (\R_+)) }
  \leq c \|    h
\|_{B^{\al -\frac2q}_{pq,0} ({\mathbb R}^n_+)}.
\end{align*}
\end{lemm}
\begin{proof}

The proof of Lemma \ref{lemma-1-2-zero} is given in Appendix \ref{appendix.lemma-1-2-zero}.
\end{proof}


\subsection{Estimate of $(V,\Pi)$}

 Let  $f = {\mathbb P}f + \na {\mathbb Q} f$ be  the  decomposition of $f$, where ${\mathbb P}$ and ${\mathbb Q}$ are the operator defined in section \ref{projection}. Note that ${\rm div}\, {\mathbb P} f =0$ and  $({\mathbb P} f)_n|_{x_n =0} =0$.  We define  $(V, \Pi_0)$  by
\begin{equation}\label{expression-v}
V_i (x,t) =\int_0^t \int_{{\mathbb R}^n_+} G_{ij}(x,y, t-\tau)
({\mathbb P} f)_j(y,\tau) dyd\tau,
\end{equation}
\begin{equation}\label{expression-p}
\Pi_0(x,t) =\int_0^t \int_{{\mathbb R}^n_+} P(x,y, t-\tau)
\cdot  ({\mathbb P} f) (y,\tau)dyd\tau,
\end{equation}
where $G$ and $P$ are defined by
\eqref{formulas-v}and \eqref{formulas-p}.
Then  $(V, \Pi_0)$ satisfies
\begin{align*}
\begin{array}{l}\vspace{2mm}
V_t - \De V + \na \Pi_0 ={\mathbb P} f, \qquad \mbox{div } V =0, \mbox{ in }
 \R_+\times (0,\infty),\\
\hspace{30mm}V|_{t=0}= 0, \qquad  V|_{x_n =0} = 0.
\end{array}
\end{align*}
(See \cite{So}.) Let $\Pi = \Pi_0 + {\mathbb Q}f$. Then, $(V, \Pi)$ is solution of \eqref{maineq-stokesh=0}.

Let $1<p<\infty$ and $0 \leq \al \leq 2$.  In Section 3 in  \cite{CK}, the authors showed that
 $V, \Pi_0 $ defined by
\eqref{expression-v} and \eqref{expression-p}   have the following estimates (using real interpolations); if $0 < \al <2$, then
\begin{align*}
\|   V \|_{  L^q(0,\infty;\dot B^{\al}_{pq}(\R_+)) } &
   \leq c \big( \|   \Gamma* {\mathbb P} f\|_{  L^q(0,\infty; \dot B^{\al}_{pq}(\R_+) )} +
  \|   \Gamma^* * {\mathbb P} f \|_{ L^q(0,\infty; \dot B^{\al}_{pq}(\R_+))} \big).
\end{align*}

Using Lemma \ref{hemoz} and  Lemma \ref{0929-1}, the following theorem holds:
Let $1 < p_1 \leq p < \infty, \, 1 < q_1 \leq q < \infty$,  $0 \leq \be \leq  \al$ and
 $ 0=\al -\be -1 +n (\frac{1}{p_1}-\frac{1}{p})+\frac2{q_1}-\frac2{q}$.
 Then, for  ${\mathcal F} \in L^{q_1} (0, \infty, \dot B^\be_{p_1q}(\R_+))$ with ${\mathcal F}|_{x_n =0} =0$, we have
\begin{align}
\label{0417-2}
\|   V \|_{L^q(0,\infty;  \dot B^{\al}_{pq}(\R_+)) } &
   \leq c \| {\mathcal F} \|_{L^{q_1}(0,\infty; \dot B^\be_{p_1 q}(\R_+))  }, \quad 0 < \al <2.
\end{align}
On the other hand, by Lemma \ref{hemoz} the following estimate hold for ${\mathbb Q}f$.
\begin{align}
\label{0910-1}
\| {\mathbb Q} f\|_{L^{q_1}(0,\infty;\dot B^\be_{p_1 q} (\R_+))} \leq c \|{\mathcal F}\|_{L^{q_1}(0,\infty;\dot B^\be_{p_1 q} (\R_+))}.
\end{align}

For $\al>\frac{1}{p}$,  the  following estimates for  $\Pi_0$  also hold.
\begin{lemm}\label{lemma-1-2}
Let $ 1 < p, q<\infty$ and  $\frac{1}{p}< \al <2$. Let  $p_1, q_1$ and $\be$  satisfy the same conditions in Lemma \ref{0929-1}.
Let $f = {\rm div}\, {\mathcal F}$ for ${\mathcal F} \in L^{q_1} (0,\infty;  \dot B^\be_{p_1q} (\R_+))$
 with ${\mathcal F}|_{x_n =0} =0$.
Then $\Pi_0=\sum_{j=1}^{n-1}D_{y_j}\Pi_{0j}+D_t\Pi_{00}$ with
\begin{align*}
  \|  \Pi_{0j}\|_{L^q (0,\infty; \dot B^{\al}_{pq} (\R_+)) }+\|  \Pi_{00}\|_{L^q (0,\infty; \dot B^{\al+1}_{pq} (\R_+)) }
  & \leq  c\| {\mathcal F}\|_{L^{q_1}(0,\infty; \dot B^\be_{p_1q} (\R_+))}.
\end{align*}
\end{lemm}
\begin{proof}

The proof of Lemma \ref{lemma-1-2} is given in Appendix \ref{appendix.lemma-1-2}.
\end{proof}

\subsection{Estimate of $(u,p)$}

Note that $(u, p)$ defined by $u = V + v$ and $p = \pi + \Pi_0 + {\mathbb Q} \,  {\rm div} \, {\mathcal F}$ are solution of \eqref{maineq-stokes}.

From \eqref{CK120-april10} and  \eqref{0417-2}, we obtain \eqref{0411-1-1}.

Let $p_1 = \pi_{00} + \Pi_{00}$, $P_j = \pi_{0j} + \Pi_{0j}$ and $P_0 = {\mathbb Q} \, {\rm div}\, {\mathcal F}$, where $\pi_{00}$ and $\pi_{0j}$ are defined in Lemma \ref{lemma-1-2-zero}, and $ \Pi_{00}$ and $\Pi_{0j}$ are defined in Lemma \ref{lemma-1-2}. Then, the corresponding pressure $p$ is decomposed by  $p = D_t p_1 + \sum_{j =1}^{j =n-1} P_j + P_0$. From \eqref{0910-1}, Lemma \ref{lemma-1-2-zero} and Lemma \ref{lemma-1-2}, we get \eqref{main0910}. Hence, we complete the proof Theorem \ref{thm-stokes}.

\section{Nonlinear problem}

\label{nonlinear}
\setcounter{equation}{0}

In this section, we would like to give  proofs of Theorem \ref{thm-navier} and Theorem \ref{maintheopressure}. For the purpose of them, we  construct approximate velocities and then
derive uniform convergence in  $L^q(0,\infty; \dot B^\al_{pq}(\R_+))$.

{\color{red}{\subsection{$p_0, \, q_0, \, p_1, \, p_2,  \, q_1$ and $ \be$} \label{subsectionpqs}


Let $(\al,p,q)$ satisfy $ 1 < p, q < \infty$, $0 < \al <2$ and $\al +1 =\frac{n}p +\frac2q$. We take $ \ep_1, \ep_2  \in (0, 1)$ satisfying
\begin{align}\label{0827-1}
\notag & \hspace{37mm} 0 < \ep_1 < \min(1,   2 -\frac2q),\\
&\max(0, 1-\al, 1 -\frac{n}p  ) < \ep_1 +\ep_2 < \min (1, 2-\al).
\end{align}

Moreover, if $1 \leq \al$, then we take $ \ep_1, \, \ep_1 \in  (0, 1) $ satisfying
\begin{align}\label{0827-2}
\notag 0< \ep_1  < \min(1,   2 -\frac2q),\\
\max(0, 2-\al, 1 -\frac{n}p) < \ep_1 +\ep_2 <1.
\end{align}

Let $\frac2{q_1} =\frac{2}q +\ep_1$, $\frac{n}{p_1} = \frac{n}p +\ep_2$,  $\frac{n}{p_0} =1 -\ep_1$, $\frac2{ q_0} =\ep_1$, $\frac{n}{p_2} = -1 +\ep_1 + \ep_2 + \frac{n}p  $ and  $\be = \al -1 +\ep_1 +\ep_2$. Then, $1<  p_0, p_1, p_2,   q_0, q_1<\infty $ and $  0<\be  $  satisfy

\begin{align*}
1<p_1< p <  p_2,  \, 1<q_1< q, \,\,\, \frac{n}{p_0} +\frac2{q_0} =1,\,\,\, \be=\al-1+n(\frac{1}{p_1}-\frac{1}{p})+\frac{2}{q_1}-\frac{2}{q}, \\
\frac1{p_1} =\frac1{p_2} +\frac1{p_0}, \,\,\, \be-\frac{n}{p_2} =\al-\frac{n}{p},\,\,\,   (0  \mbox{  or  } 1)<\be< \al ,\,\,\, \frac{1}{q_1}=\frac{1}{q}+\frac{1}{q_0}.
\end{align*}

}}

\subsection{Approximating solutions}

Let $(u^1,p^1)$ be the solution of the Stokes equations
\begin{align}
\begin{array}{l}\vspace{2mm}
u^1_t - \De u^1 + \na p^1 =0, \qquad \mbox{div } u^1 =0, \mbox{ in }
 \R_+\times (0,\infty),\\
\hspace{30mm}u^1|_{t=0}= h, \qquad  u^1|_{x_n =0} = 0.
\end{array}
\end{align}
Let $m\geq 1$.
After obtaining $(u^1,p^1),\cdots, (u^m,p^m)$ construct $(u^{m+1}, p^{m+1})$ which satisfies the following equations
\begin{align}
\label{maineq5}
\begin{array}{l}\vspace{2mm}
u^{m+1}_t - \De u^{m+1} + \na p^{m+1} =f^m, \qquad \mbox{div } u^{m+1} =0, \mbox{ in }
 \R_+\times (0,\infty),\\
\hspace{30mm}u^{m+1}|_{t=0}= h, \qquad  u^{m+1}|_{x_n =0} = 0,
\end{array}
\end{align}
where $f^m=-\mbox{div}(u^m\otimes u^m)$.

\subsection{Uniform boundedness in $L^{q_0}(0,\infty;L^{p_0}(\R_+))$}
\label{uniform1}


Let $1<p_0,q_0<\infty$ with
\begin{equation}
\label{con-1}
\frac{n}{p_0} + \frac2{q_0} =1.
\end{equation}
(Observe that $n<p_0<\infty,\ 2<q_0<\infty$ and $-1+\frac{n}{p_0}=-\frac{2}{q_0}$.) From  Theorem \ref{thm-stokes}, we have
\begin{align}
\label{uc1-1}
\| u^{1}\|_{L^{q_0}(0,\infty;  L^{p_0} (\R_+))}
 \leq c_0  \|h\|_{  \dot B^{-1+\frac{n}{p_0}}_{p_0q_0,0}({\mathbb
R}^{n}_+)}:=N_0.
\end{align}
From Theorem \ref{thm-stokes}, taking $p_1 =\frac{p_0}2$, $q_1 =\frac{q_0}2$ and $\al =\be =0$, we have 
\begin{align}\label{1011-1-2}
\notag  \| u^{m+1}\|_{L^{q_0} (0,\infty; L^{p_0} (\R_+))} & \leq  c \big( \| h\|_{\dot B^{-\frac2{q_0}}_{p_0q_0,0} (\R_+)}  +    \| u^{m} \otimes u^{m} \|_{L^{\frac{q_0}2}(0,\infty; L^{\frac{p_0}2} (\R_+))} \big)\\
  & \leq  c_1 \big( \| h\|_{\dot B^{-\frac{2}{q_0}}_{p_0q_0,0} (\R_+)}  +    \| u^{m}  \|^2_{L^{q_0}(0,\infty; L^{p_0} (\R_+))} \big).
\end{align}

Under the hypothesis  $\|u^m\|_{L^{q_0}(0, \infty; L^{p_0} (\R_+) )}\leq M_0$,  \eqref{1011-1-2} leads to the estimate
\[
\|u^{m+1}\|_{L^{q_0}(0, \infty; L^{p_0} (\R_+) )}\leq c_1 \big( N_0+  M_0^2 \big).
\]
Choose $M_0$ and $N_0$  so small that
\begin{align}\label{0508-2}
 M_0  \le \frac{1}{2c_1}\mbox{ and }N_0<\frac{M_0}{2c_1}.
\end{align}
By the mathematical induction argument, we conclude
\begin{align}\label{0509-1}
\|u^{m}\|_{L^{q_0}(0, \infty; L^{p_0}(\R_+) )}\leq M_0 \,  \mbox{ for all } \, m=1,2\cdots.
\end{align}

\subsection{Uniform boundedness in $L^{q}(0,\infty;\dot B^\al_{pq}(\R_+))$}
\label{uniform2}

Let
\begin{equation}
\label{con-2}
\frac{n}p + \frac2q =1 +\al.
\end{equation}
(Observe that $\frac{n}{\al+1}<p<\infty$, $\frac{2}{\al+1}<q<\infty$ and $-1+\frac{n}{p}=\al-\frac{2}{q}$.) From  Theorem \ref{thm-stokes}, we have
\begin{align}
\label{0509-1-1}
\| u^{1}\|_{L^{q}(0,\infty;  \dot B^\al_{pq} (\R_+))}
& \leq c_2 \|h\|_{  \dot B^{-1+\frac{n}{p}}_{pq,0}({\mathbb
R}^{n}_+)}:=N.
\end{align}

Let $p_0, \, q_0, \, p_1, \, q_1$ and $ \be$  are constants defined in Subsection \ref{subsectionpqs}.
We  apply  Theorem \ref{thm-stokes} and Lemma \ref{0510prop}, respectively, to obtain
\begin{align}\label{1011-1-2-1}
 \| u^{m+1}\|_{L^{q} (0,\infty; \dot{B}^\al_{pq} (\R_+))} & \leq  c \big( \| h\|_{\dot B^{-1+\frac n{p}}_{pq,0} (\R_+)}  +    \| u^{m} \otimes u^{m} \|_{L^{{q_1}}(0,\infty; \dot{B}^\be_{p_1q} (\R_+))} \big)
\end{align}
and
\begin{align}\label{0510-3-1}
\notag\|(u^m\otimes u^m)\|_{L^{q_1} (0,\infty; \dot B^{\be}_{p_1q}(\R_+)) }
& \leq c \|u^m \|_{L^{q} (0,\infty; \dot B^{\be }_{p_2q}(\R_+)) } \|u^m  \|_{L^{ q_0} (0,\infty;L^{p_0}(\R_+)) }\\
& \leq c \|u^m \|_{L^{q} (0,\infty; \dot B^{\al }_{pq}(\R_+)) } \|u^m  \|_{L^{ q_0} (0,\infty;L^{p_0}(\R_+)) }.
\end{align}
From \eqref{1011-1-2-1}-\eqref{0510-3-1}, we have
\begin{align}\label{eq0412-3}
 \| u^{m+1}\|_{L^{q}(0,\infty; \dot B^\al_{pq}(\R_+))}
  \leq   c_1 \big( N
 +  \|u^m  \|_{L^{ q_0} (0,\infty;L^{p_0}(\R_+)) } \|u^m\|_{ L^{q}(0,\infty;\dot B^\al_{pq}({\mathbb R}^n_+))} \big).
\end{align}

Under the hypothesis
$\|u^m\|_{L^q(0, \infty; \dot B^\al_{pq}(\R_+) )}\leq M$,
\eqref{eq0412-3} leads to the estimate
\[
\|u^{m+1}\|_{L^q(0, \infty; \dot B^\al_{pq}(\R_+) )}\leq c_1 \big( N+  M_0M \big).
\]
Choose $M_0$, $N_0$ so small that
\begin{equation}
\label{0508-2-2} M_0 \leq  \frac1{2c_1}, \qquad  N_0\leq \frac{M_0}{2c_1}.
\end{equation}
and $M$ so large that
\[
 2c_1 N\leq M.\]
By the mathematical induction argument, we conclude
\begin{equation}
\label{0509-1-1}
\|u^{m}\|_{L^q(0, \infty; \dot B^\al_{pq}(\R_+) )}\leq M \mbox{ for all }m=1,2\cdots.
\end{equation}

\subsection{Uniform convergence}

Let $U^m=u^{m+1}-u^m$ and $P^m=p^{m+1}-p^m$.
Then, $(U^m,P^m)$ satisfy the equations
\[
\begin{array}{l}\vspace{2mm}
U^m_t - \De U^m + \na P^m =-\mbox{div}(u^m\otimes U^{m-1}+U^{m-1}\otimes u^{m-1}), \qquad \mbox{div } U^{m} =0, \mbox{ in }
 \R_+\times (0,\infty),\\
\hspace{30mm}U^{m}|_{t=0}= 0, \qquad  U^{m}|_{x_n =0} =0.
\end{array}
\]

Recall  the uniform estimates \eqref{0509-1} and \eqref{0509-1-1} for the approximate solutions. From  Theorem \ref{thm-stokes} and  Lemma \ref{0510prop} we have
\begin{align}\label{0425-3}
\notag\| U^m\|_{L^{q_0}(0, \infty; L^{p_0} (\R_+))}
& \leq c \big( \| u^{m-1}\|_{L^{q_0}(0, \infty; L^{p_0}(\R_+))}  + \| u^{m}\|_{L^{q_0}(0, \infty; L^{p_0}(\R_+))}  \big) \| U^{m-1}\|_{L^{q_0}(0, \infty; L^{p_0}(\R_+))}\\
 & \leq c_5 M_0 \| U^{m-1}\|_{L^{q_0}(0, \infty; L^{p_0}(\R_+))},
\end{align}
%
%
and
\begin{align}\label{0425-2}
\notag \|U^m\|_{L^q(0, \infty; \dot B^\al_{pq}(\R_+) )}
&\leq c\|u^m\otimes U^{m-1}+U^{m-1}\otimes u^{m-1}\|_{L^{q_1}(0,\infty;\dot B^{\be}_{p_1q_1} ({\mathbb R}^n_+))}\\
\notag &\leq c \big(\|u^m\|_{L^q(0, \infty; \dot B^\al_{pq}(\R_+) )} + \|u^{m-1}\|_{L^q(0, \infty; \dot B^\al_{pq}(\R_+) )} \big) \|U^{m-1}\|_{L^{ q_0}(0, \infty;  L^{p_0}(\R_+) )}\\
    \notag &\qquad + c\big( \|u^m\|_{L^{ q_0}(0, \infty; L^{p_0}(\R_+) )}  + \|u^{m-1}\|_{L^{ q_0}(0, \infty; L^{p_0}(\R_+) )} \big)\|U^{m-1}\|_{L^q(0, \infty; \dot B^\al_{pq}(\R_+) )}  \big)\\
     &\leq c_6 M \|U^{m-1}\|_{L^{q_0}(0, \infty;  L^{p_0}(\R_+) )}    +   c_6M_0\|U^{m-1}\|_{L^q(0, \infty; \dot B^\al_{pq}(\R_+) )}.
\end{align}
Here, $\al,\be,p_,p_0,p_1,q,q_0,q_1,N_0,M_0$ are the same numbers as in the previous subsection to satisfy the condition  \eqref{0508-2-2}.
Here, we take the constant $c_6$ greater than $c_5$, that is,
\begin{align}
c_6>c_5.
\end{align}

From \eqref{0425-3}, if $c_5M_0<1$, then $\sum_{m=1}^\infty\| U^m\|_{L^{q_0}(0, \infty; L^{p_0} (\R_+))}$ converges, that is,
$$\sum_{m=1}^\infty U^m\mbox{ converges in }L^{q_0}(0, \infty; L^{p_0} (\R_+)).$$

Take $A>0$ satisfying $A(c_6-c_5)M_0\geq c_6 M$. Then from \eqref{0425-3} and \eqref{0425-2} it holds that
\begin{align*}
 \|U^m\|_{L^q(0, \infty; \dot B^\al_{pq}(\R_+) )}+A\| U^m\|_{L^{q_0}(0, \infty; L^{p_0} (\R_+))}\\
 \leq c_6M_0(\|U^{m-1}\|_{L^q(0, \infty; \dot B^\al_{pq}(\R_+) )}+A\| U^{m-1}\|_{L^{q_0}(0, \infty; L^{p_0} (\R_+))})
 \end{align*}
 Again if $c_6M_0<1$, then $\sum_{m=1}^\infty(\|U^m\|_{L^q(0, \infty; \dot B^\al_{pq}(\R_+) )}+A\| U^m\|_{L^{q_0}(0, \infty; L^{p_0} (\R_+))})$ converges. This implies that
 $\sum_{m=1}^\infty\|U^m\|_{L^q(0, \infty; \dot B^\al_{pq}(\R_+) )}$ converges, that is, $$\sum_{m=1}^\infty U^m\mbox{ converges in }L^q(0, \infty; \dot B^\al_{pq}(\R_+) ).$$

Therefore, if $M_0$ satisfies the condition \eqref{0508-2-2} with the additional conditions
\begin{align}
M_0< \frac{1}{c_6},
\end{align}
then  $u^m=u^1+\sum_{k=1}^mU^{k}$ converges to $u^1+\sum_{k=1}^\infty U^{k}$ in  $L^q(0, \infty; \dot B^\al_{pq}(\R_+) ) \cap L^{q_0}(0, \infty; L^{p_0} (\R_+))$.
Set $u:=u^1+\sum_{k=1}^\infty U^{k}.$

\subsection{Existence}

Let $u$ be the same one constructed in  the previous Section.
By the lower semi continuity we have
\[
\|u\|_{L^{q_0}(0,\infty;L^{p_0}(\R_+))}\leq \limsup_{n\geq 1}\|u_n\|_{L^{q_0}(0,\infty;L^{p_0}(\R_+))}\leq M_0,\]
and
\[
\|u\|_{L^{q}(0,\infty;\dot{B}^{\al}_{pq}(\R_+))}\leq \limsup_{n\geq 1}\|u_n\|_{L^{q}(0,\infty;\dot{B}^{\al}_{pq}(\R_+))}\leq M.\]

In this section, we will show that $u$ satisfies weak formulation of Navier-Stokes equations \eqref{weaksolution-NS}, that is, $u$ is a weak solution of Navier-Stokes equations \eqref{maineq2} with appropriate distribution $p$.
Let $\Phi\in C^\infty_{0}(\overline{\R_+} \times [0,\infty))$ with $\mbox{div }\Phi=0$ and $\Phi|_{x_n=0}=0$.
Observe that
\begin{align*}
-\int^\infty_0\int_{\R_+} u^{m+1}\cdot \Delta\Phi dxdt&=\int^\infty_0\int_{\R_+}u^{m+1}\cdot \Phi_t+(u^m\otimes u^m): \nabla \Phi dxdt  +<h,\Phi(\cdot,0)>.
\end{align*}
Now, send $m$ to the  infinity, then,  $u^m\rightarrow u$ in $L^{q_0} (0, \infty; L^{p_0} (\R_+ ))$.   Since $n < p_0$ and $2 < q_0$,
 $u^m\otimes u^m\rightarrow u\otimes u$ in  $L^1_{loc} ( \R_+ \times [0, \infty ))$.
Hence,
we have the identity
\begin{align*}
-\int^\infty_0\int_{\R_+}u\cdot \Delta \Phi dxdt&=\int^\infty_0\int_{\R_+}u\cdot \Phi_t+(u\otimes u): \nabla\Phi dxdt +<h,\Phi(\cdot,0)>.
\end{align*}
Therefore  $u$ is a weak solution of Navier-Stokes equations \eqref{maineq2}.

\subsection{Uniqueness in space $L^{q_0}(0,\infty;L^{p_0}(\R_+))$}
Let 
  $ u_1\in  L^{q_0}(0,\infty; L^{p_0}(\R_+))$ be  another weak solution of Naiver-Stokes equations \eqref{maineq2} with pressure $p_1$. Then,
 $(u-u_1,p-p_1)$ satisfies the equations
\begin{align*}
(u-u_1)_t - \De (u-u_1) + \na (p-p_1)& =-\mbox{div}(u\otimes (u-u_1)+(u-u_1)\otimes u_1)\mbox{ in }
 \R_+\times (0,\infty), \\
 {\rm div} \, (u-u_1)& =0,
 \mbox{ in }\R_+\times (0,\infty),\\
 (u-u_1)|_{t=0}= 0, &\quad (u-u_1)|_{x_n =0} =0.
\end{align*}

Applying  the estimate of Theorem \ref{thm-stokes} in \cite{CJ2} to the above Stokes equations,  we have
\begin{align*}
\| u-u_1\|_{L^{q_0}(0,\tau; L^{p_0}{\mathbb R}^n_+ ))}
 \leq c \|u\otimes (u-u_1)+(u-u_1)\otimes u_1 \|_{L^{\frac{q_0}2}(0,\tau; L^{\frac{p_0}2}{\mathbb R}^n_+ ))}\\
 \leq c_5 ( \|u\|_{L^{q_0}(0,\tau; L^{p_0}({\mathbb R}^n_+ ))}+\|u_1\|_{L^{q_0}(0,\tau; L^{p_0}({\mathbb R}^n_+ ))})\| u-u_1\|_{L^{q_0}(0,\tau; L^{p_0}({\mathbb R}^n_+ ))}, \ \tau <\infty.
 \end{align*}
 Since  $u,u_1\in L^{q_0}(0,\infty; L^{p_0}({\mathbb R}^n_+ ))$, there is $ 0 <\de$ such that if $\tau_2 -\tau_1 \leq \de$ for $\tau_1 < \tau_2$, then
 \[ \|u\|_{L^{q_0}(\tau_1,\tau_2; L^{p_0}({\mathbb R}^n_+ ))}+\|u_1\|_{L^{q_0}(\tau_1,\tau_2; L^{p_0}({\mathbb R}^n_+ ))}<\frac{1}{c_5}\]
 (See Radon Nikodym theorem in \cite{rudin} for the reference).

 Hence, we have
\[
 \| u-u_1\|_{L^{q_0}(0,\delta; L^{p_0}({\mathbb R}^n_+ ))}
 <\| u-u_1\|_{L^{q_0}(0,\delta; L^{p_0}({\mathbb R}^n_+ ))}.\]
 This implies that
 $\| u-u_1\|_{L^{q_0}(0,\delta; L^{p_0}({\mathbb R}^n_+ ))}=0$, that is, $u\equiv u_1$ in $\R_+\times (0,\de]$.
Observe that  $u-u_1$ satisfies the Stokes equations
 \begin{align*}
(u-u_1)_t - \De (u-u_1) + \na (p-p_1)& =-\mbox{div}(u\otimes (u-u_1)+(u-u_1)\otimes u_1)\mbox{ in }
 \R_+\times (\de,\infty), \\
 \mbox{div } (u-u_1)& =0
 \mbox{ in }\R_+\times (\de,\infty),\\
 (u-u_1)|_{t=\de}= 0, &\quad (u-u_1)|_{x_n =0} =0.
\end{align*}
Again, applying  the estimate of Theorem \ref{thm-stokes}  in \cite{CJ2} to the above Stokes equations,  we have
\begin{align*}
\| u-u_1\|_{L^{q_0}(\de,2\de; L^{p_0}({\mathbb R}^n_+ ))}
& \leq c_5 ( \|u\|_{L^{q_2}(\de,2\de; L^{p_0}({\mathbb R}^n_+ ))}+\|u_1\|_{L^{q_0}(\de,2\de; L^{p_0}({\mathbb R}^n_+ ))})\| u-u_1\|_{L^{q_0}(\de,2\de; L^{p_0}({\mathbb R}^n_+ ))}\\
 & < \| u-u_1\|_{L^{q_0}(\de,2\de; L^{p_0}({\mathbb R}^n_+ ))}.
 \end{align*}
 This implies that
 $\| u-u_1\|_{L^{q_0}(\de,2\de; L^{p_0}({\mathbb R}^n_+ ))}=0$, that is, $u\equiv u_1$ in $\R_+\times [\de,2\de]$.
After iterating this procedure  infinitely, we obtain  the conclusion that $u=u_1$ in $\R_+\times (0,\infty)$.
Therefore, we conclude the proof of  the global in time uniqueness.

\subsection{Pressure  Estimate}

In this section, we prove Theorem \ref{maintheopressure}.

Let $\al,\be,p,p_1,p_0,q,q_1,q_0$ be the same as in  section \ref{uniform2}. If $\al>\frac{1}{p}$, then by   \eqref{main0910},  there is $P_0,P_j,p_1$ so that  $p=D_tp_0+\sum_{j=1}^{n-1}D_{x_j}P_j+p_1$ with
\begin{align*}
\| p_1\|_{L^{q_1}(0,\infty; \dot B^{\be}_{p_1q_1}(\R_+)) }+\sum_{j=1}^{n-1}\| P_j\|_{L^{q}(0,\infty; \dot B^\al_{pq}(\R_+)) }+\| P_0\|_{L^{q}(0,\infty; \dot B^{\al+1}_{pq}(\R_+)) }\\
  \leq   c \big( \|h\|_{ \dot B^{\al-\frac{2}{q}}_{pq,0}({\mathbb R}^{n}_+)}
+ \|(u\otimes u)\|_{L^{q_1} (0,\infty; \dot B^{\be}_{p_1q} (\R_+)) } \big)\\
\leq c_1 \big(\|h\|_{ \dot B^{\al-\frac{2}{q}}_{pq,0}({\mathbb R}^{n}_+)} +    \|u\|_{L^{q_0} (0,\infty; L^{p_0} (\R_+))} \|u\|_{L^q (0,\infty; \dot B^\al_{pq } (\R_+))} \big).
\end{align*} 
This completes the proof of the first pressure estimates.

\appendix
\setcounter{equation}{0}

\section{Proof of Lemma \ref{hemoz}}
\label{appendix.hemoz}

\begin{lemm}\label{lemma0523-3}

Let  $f = {\rm div} {\mathcal F} $(Here ${\mathcal F}=(F_{ki})_{k,i=1,\cdots, n}$, $f_i=D_{x_k}F_{ki}$) with    ${\mathcal F}|_{x_n =0}=0$.   Let $F_{k}=(F_{k1},\cdots,F_{kn})$. Then,
\begin{align*}
{\mathbb Q}_1  f(x) &=\sum_{ k \neq n} D_{x_k} {\mathbb Q}_1 F_k(x) +  D_{x_n} A(x),\\
{\mathbb Q}_2 f(x) &=
\sum_{  k \neq  n}  D_{x_k} {\mathbb Q}_2 F_k (x) - \sum_{ k \neq n} D^2_{x_k} B(x),
\end{align*}
where
\begin{align*}
A(x) & = -\int_{\R_+}\na_y \big( N(x-y) + N(x-y^*) \big)\cdot F_n(y) dy,\\
B(x) & = \int_{\Rn} N(x'-y',x_n) A(y',0) dy'.
\end{align*}

\end{lemm}

\begin{proof}

From \eqref{0427-1}, we have
\begin{align}\label{0424-1}
\notag {\mathbb Q}_1  f(x) 
\notag &= - \sum_{ k  \neq n} D_{x_k}\int_{\R_+} \na_y \big( N(x-y) - N(x-y^*) \big) \cdot F_k(y) dy\\
\notag  &  \qquad-  D_{x_n} \int_{\R_+}\na_y \big( N(x-y) + N(x-y^*) \big)\cdot F_n(y) dy\\
&:=\sum_{ k \neq n} D_{x_k} {\mathbb Q}_1 F_k(x) +  D_{x_n} A(x).
\end{align}

Since $\De A = {\rm div} \, F_n = f_n$, from \eqref{0424-1}, we have  $D_{y_n}{\mathbb Q}_1f (y) = D_{y_n} \sum_{k \neq n} D_{y_k} {\mathbb Q}_1 F_k(y)  +  f_n (y) - \De' A(y) $. Hence, we have
\begin{align}\label{0501-1}
\notag \int_{\Rn} N(x' -y', x_n) D_{y_n} {\mathbb Q}_1f (y',0) dy'  
\notag & = \sum_{  k \neq  n}  D_{x_k} \int_{\Rn} N(x'-y',x_n)   D_{y_n} {\mathbb Q}_1 F_k(y',0)  dy'\\
\notag &\qquad  + \int_{\Rn} N(x'-y',x_n)f_n(y',0) dy'\\
&\qquad  + \sum_{ k \neq n} D_{x_k} \int_{\Rn} N(x'-y',x_n) D_{y_k}A(y',0) dy'.
\end{align}
Hence, from \eqref{0427-2} and  \eqref{0501-1}, we have
\begin{align} \label{Q-2-1}
\notag{\mathbb Q}_2 f(x) &=
-\sum_{  k \neq  n}  D_{x_k} \int_{\Rn} N(x'-y',x_n)  D_{y_n} {\mathbb Q}_1 F_k(y',0)  dy'\\
&\qquad  - \sum_{ k \neq n} D_{x_k} \int_{\Rn} N(x'-y',x_n) D_{y_k}A(y',0) dy'.
\end{align}

%
\end{proof}

Now, we will show that for nonnegative integers $k$,
\begin{align}\label{0413-1}
\| {\mathbb Q} f\|_{\dot W^k_{p} (\R_+)} \leq c \|{\mathcal F}\|_{\dot W^k_p (\R_+)}.
\end{align}
Using the property of complex interpolation, \eqref{0413-1} implies lemma \ref{hemoz}.

From Lemma \ref{lemma0523-3}, 
 we have
\begin{align*}
\| {\mathbb Q} f\|_{\dot W^k_p (\R_+)} &\leq c \big(
    \sum_{k\neq n}\| D_{x_k} {\mathbb Q}_1F_k \|_{\dot W^k_p(\R_+)}  +  \sum_{k\neq n}\| D_{x_k} {\mathbb Q}_2 F_k \|_{\dot W^k_p(\R_+)} \\
    &\qquad+ \|D_{x_n} A\|_{\dot W^k_p (\R_+)}  +  \sum_{k\neq n} \|D^2_{x_k} B \|_{\dot W^k_p(\R_+)} \big),
\end{align*}
where ${\mathbb Q}_1$ and ${\mathbb Q}_2$ are defined in \eqref{0427-1} and \eqref{0427-2}, and $A$ and $B$ are defined in Lemma \ref{lemma0523-3}, respectively. From   Lemma \ref{lemma0929-22}, we have
\begin{align*}
\| D_{x_k} {\mathbb Q}_1 F_k\|_{\dot W^k_p (\R_+)}  \leq c \|  F\|_{ \dot W^k_p (\R_+)}, \qquad \| D_{x_n} A \|_{ \dot W^k_p(\R_+)}   \leq c \|  F\|_{ \dot W^k_p(\R_+)}.
\end{align*}
From \eqref{Poisson},  Lemma \ref{trace} and  Lemma \ref{lemma0929-22} continuously, we get
\begin{align*}
\|D^2_{x_k} B \|_{\dot W^k_p (\R_+)} \leq  c \| A|_{x_n =0} \|_{\dot B^{k+1  -\frac1p}_{pp} (\Rn)} \leq  c \| D_x A \|_{\dot W^k_p (\R_+)} \leq c \| F_n\|_{\dot W^k_p (\R_+)}.
\end{align*}
From \eqref{Poisson}, we have
  \begin{align*}
 \| D_{x_k}  {\mathbb Q}_2 F_k \|_{\dot W^k_p (\R_+)} \leq c  \|\big( F_k - \na {\mathbb Q}_1F_k \big)_n|_{x_n =0} \|_{\dot B^{k-\frac1p}_{pp} (\Rn)}.
  \end{align*}
If $k > 0$, then by (1) of  Lemma \ref{trace}, we have
 \begin{align*}
\|\big( F_k - \na {\mathbb Q}_1F_k \big)_n|_{x_n =0} \|_{\dot B^{k -\frac1p}_{pp} (\Rn)} \leq  \|F_k - \na {\mathbb Q}_1F_k  \|_{\dot W^k_p (\R_+)}\leq c \|F   \|_{\dot W^k_p (\R_+)}.
 \end{align*}
 If $k =0$, then since $F_k - \na {\mathbb Q}_1F_k  $ is  divergence free in $\R_+$, by (2) of Lemma \ref{trace}, we have
 \begin{align*}
\|\big( F_k - \na {\mathbb Q}_1F_k \big)_n|_{x_n =0} \|_{\dot B^{ -\frac1p}_{pp} (\Rn)} \leq  \|F_k - \na {\mathbb Q}_1F_k  \|_{ L^p (\R_+)}\leq c \|F   \|_{L^p (\R_+)}.
 \end{align*}
Therefore we obtain the estimate \eqref{0413-1}.

Using the properties of real interpolation an complex interpolation (see \eqref{interpolation1}), we complete the proof of Lemma  \ref{hemoz}.

\section{Proof of Lemma \ref{0929-1}}

\label{appendix0292-1}

Since the proofs will be done  by the same way, we  prove only the case of $ \Ga^* * {\mathbb P}f$.

Recall that  $({\mathbb P}\, f)_j = f_j - D_{x_j} {\mathbb Q}f=\mbox{div}F_j-D_{x_j} {\mathbb Q}f$. For $1 \leq j \leq n-1$  we have
\begin{align}\label{0421-1}
\notag \Ga^* * ({\mathbb P}f )_j (x,t) &= -\int_0^t \int_{{\mathbb R}^n_+}
 \nabla_{y} \Ga(x-y^*, t-\tau) \cdot  F_{j}(y,\tau) dyd\tau \\
&\quad +  \int_0^t \int_{{\mathbb R}^n_+}
D_{y_j} \Ga(x-y^*, t-\tau) {\mathbb Q}f (y,\tau) dyd\tau.
\end{align}
From the relation  $D_{x_n}^2 A = -\De' A + \sum_{k=1}^nD_{x_k} F_{nk}$ and    from Lemma \ref{lemma0523-3}, we have  the identity $({\mathbb P}\, f)_n= \Delta' A+D_{x_n}\Delta'B-\sum_{k\neq n}D_{x_n}D_{x_k}QF_k$. Hence  we have
\begin{align}\label{0421-2}
\notag \Ga^* * ({\mathbb P}\, f)_n (x,t)
&= 
 \sum_{1 \leq k \leq n-1}   \int_0^t \int_{{\mathbb R}^n_+}  D_{y_k}\Ga(x-y^*, t-\tau)D_{y_n} {\mathbb Q}  F_k(y,\tau) dyd\tau\\
\notag &\quad +  \int_0^t \int_{{\mathbb R}^n_+} \na_{x'} \Ga(x-y^*, t-\tau)\cdot \na_{y'}A(y,\tau) dyd\tau\\
&\quad +  \int_0^t \int_{{\mathbb R}^n_+} \na'_{x'} \Ga(x-y^*, t-\tau)\cdot \na_{y'} D_{y_n}B(y,\tau) dyd\tau.
\end{align}

Using \eqref{0421-1}, \eqref{0421-2}, Young's inequality and the proof of  Lemma \ref{hemoz} ($\al =0$),  for  $t_1 \leq t_0$, we have
\begin{align}\label{0508-1}
\notag \|\Ga^* * ({\mathbb P}\, f)(t)\|_{L^{p}(\R_+) }
& \leq  c \int_0^t (t-s)^{-\frac12 -\frac{n}2 (\frac1{t_1} -\frac1{p})}( \sum_{j=1}^{n-1}\| F_j(s)\|_{L^{t_1}(\R_+)} + \|{\mathbb Q} f(s)\|_{L^{t_1}(\R_+)}\\
\notag&\qquad+\sum_{j=1}^{n-1}\| D_{y_n}{\mathbb Q} F_j(s)(s)\|_{L^{t_1}(\R_+)}+\| \nabla A(s)\|_{L^{t_1}(\R_+)}+\|\nabla^2B(s)\|_{L^{t_1}(\R_+)}) ds \\
&\leq  c \int_0^t (t-s)^{-\frac12 -\frac{n}2 (\frac1{t_1} -\frac1{p})} \| {\mathcal F}(s)\|_{L^{t_1}(\R_+)} ds.
\end{align}
Hence, from Hardy-Littlewood-Sobolev inequality (see \cite{Tr}), 
we have
\begin{align}
\label{h1}
\|\Ga^* * ({\mathbb P}\, f) \|_{L^{s_0}(0,\infty; L^{p} (\R_+))} \leq c \| {\mathcal F}\|_{L^{q_1}(0, \infty;  L^{t_1}(\R_+))}
\end{align}
when
\begin{align}\label{0622-1}
\frac1{s_0} + 1 = \frac1{q_1} + \frac12  +\frac{n}2 (\frac1{t_1} -\frac1{p}), \quad q_1 \leq s_0,\quad  t_1\leq p.
\end{align}
Hence, we proved   Lemma \ref{0929-1} for $\al =0$.

As the same reason to \eqref{h1}, we have
\begin{align}\label{h2}
\notag\|\nabla \Ga^* * ({\mathbb P}\, f) \|_{L^{s_0}(0,\infty; L^{p} (\R_+))} & \leq c \big( \|  f\|_{L^{q_1}(0, \infty;  L^{t_1}(\R_+))} + \|\nabla \, {\mathbb Q} \, f\|_{L^{q_1}(0, \infty;  L^{t_1}(\R_+))}\big)\\
& \leq c \|\nabla {\mathcal F}\|_{L^{q_1}(0, \infty;  L^{t_1}(\R_+))}.
\end{align}
Using the property of  real  interpolation to \eqref{h1} and \eqref{h2}, we have
\begin{align}
\label{h3}
\| \Ga^* * ({\mathbb P}\, f) \|_{L^{s_0}(0,\infty; \dot{B}^\beta_{pq} (\R_+))} \leq c \| {\mathcal F}\|_{L^{q_1}(0, \infty;  \dot{B}^\beta_{t_1q}(\R_+))}
\end{align}
for $0<\beta<1$ and $ t_1 \leq p, q_1\leq s_0$ satisfying  \eqref{0622-1}.


By parabolic type's Calderen-Zygmund Theorem and Lemma \ref{hemoz}, we have
\begin{align}\label{0418-3}
\notag \| \nabla^2\Ga^* * ({\mathbb P}\, f)\|_{L^{q_1}(0,\infty; L^{p} (\R_+) } \leq  c \| {\mathbb P}\, f\|_{L^{q_1}(0,\infty; L^{p}(\R_+)) }\\
\leq  c \big(  \|   f\|_{L^{q_1}(0,\infty; L^{p}(\R_+)) }  + \| \na  {\mathbb Q}\, f\|_{L^{q_1}(0,\infty; L^{p}(\R_+)) }    \big)
 \leq   c \|  {\mathcal F}\|_{L^{q_1}(0,\infty; \dot W^1_{p}(\R_+))}
\end{align}
and
\begin{align}\label{h4}
 \| \nabla \Ga^* * ({\mathbb P}\, f)\|_{L^{q_1}(0,\infty; L^{p} (\R_+) } \leq    c \|  {\mathcal F}\|_{L^{q_1}(0,\infty; L^{p}(\R_+))}.
\end{align}
By the property of complex interpolation  to \eqref{0418-3} and \eqref{h4} we have
\begin{align}\label{h5}
 \| \Ga^* * ({\mathbb P}\, f)\|_{L^{q_1}(0,\infty; \dot B^{\be+1}_{pq} (\R_+) } \leq    c \| {\mathcal  F}\|_{L^{q_1}(0,\infty; \dot B^\be_{pq}(\R_+))},\ 0<\be<1.
\end{align}
Let $\te = 1+\be - \al$, $\frac{1 -\al +\be}{s_0} = \frac1q -\frac{\al -\be}{q_1}$ and $\frac{1 -\al +\be}{p_0} = \frac1{p_1} -\frac{\al -\be}{p}$. By the assumptions of lemma \ref{0929-1}, we get $0 \leq \te  \leq  1$ and $ t_1  \leq  p, q_1 \leq  s_0$ satisfy  \eqref{0622-1}. By the complex interpolation with exponent $\te$ to \eqref{h3} and \eqref{h5}, for $1 < \be <\al <2$, we have
\begin{align*}
\|\Ga^* * ({\mathbb P}\, f) \|_{L^q(0, \infty; \dot B^\al_{p}(\R_+))  }  \leq c \| {\mathcal F} \|_{L^{q_1}(0,\infty; \dot B^{\be}_{p_1 }(\R_+))  }.
\end{align*}
This completes the proof of   Lemma \ref{0929-1}.

\section{Proof of Lemma \ref{lemma-1-2-zero}}
\label{appendix.lemma-1-2-zero}

Recalling the formula
\eqref{expression-p-zero} with \eqref{formulas-p}, we split $\pi$ in two terms,
i.e. $\pi(x,t) = \pi_1 (x,t) + \pi_2(x,t)$, where
\begin{align*}
\pi_1(x,t) =4 \sum_{1 \leq j \leq n-1}\int_{{\mathbb R}^{n-1}}
\frac{\pa^2 N(x'-z',x_n)}{\pa x_j \pa x_n}  \int_{{\mathbb
R}^n_+} \Ga(z' -y', y_n, t ) h_j(y, \tau) dy dz',
\end{align*}
\begin{align*} 
\pi_2(x,t) =4 \sum_{1 \leq j \leq n-1}\int_{{\mathbb R}^{n-1}}
\frac{\pa N(x'-z',x_n)}{ \pa x_j}  \int_{{\mathbb R}^n_+}
\frac{\pa \Ga(z' -y', y_n, t)}{\pa y_n}h_j(y,\tau) dy dz'.
\end{align*}

Note that $\pi_1$ is represented by
\begin{align*}
\pi_1(x,t) = 4 \sum_{1 \leq j \leq n-1}  D_{x_j}  D_{x_n} N (\Ga *_t
h_j|_{x_n =0})(x)= 4 \sum_{1 \leq j \leq n-1} D_{x_j} \pi_{1j}(x,t),
\end{align*}
where $\pi_{1j}(x,t) = 4  D_{x_n} N (\Ga *_t
h_j|_{x_n =0}).$
From \eqref{Poisson},   Lemma \ref{trace} and Lemma \ref{gauss.equiv}, for  $\al>\frac{1}{p}$ we have
\begin{align}
\notag  \| \pi_{1j} \|_{ L^q(0,\infty; \dot B^{\al}_{pq}(\R_+))}
 \leq c
            \|  \Ga_t * h_j|_{x_n =0}  \|_{ L^q(0,\infty; \dot B^{\al -\frac1p}_{pq} ({\mathbb R}^{n-1})}\\
  \leq  c  \|   \Ga_t * h_j \|_{ L^q(0,\infty; \dot B^{\al}_{pq}(\R_+))}
  \leq c \| h_j\|_{\dot{B}^{\al-\frac{2}q}_{pq,0}(\R_+) }.
\end{align}

By integrating by parts, by the identity $D_{y_n}h_n=-\sum_{j=1}^{n-1}D_{y_j}h_j$ and by the identity $D_{y_n}^2 \Gamma_t*h_n=D_t\Gamma_t*h_n-\Delta_{y'}\Gamma_t*h_n$, $\pi_2$ can be rewritten by
\begin{align*}
\pi_2&=-4\int_{\Rn}N(x'-z',x_n)\int_{\R_+}D_{y_n}\Gamma(z'-y',y_n,t)D_{y_n}h_n(y)dydz'\\
&=D_t\Big(4\int_{\Rn}N(x'-z',x_n)\int_{\R_+}\Gamma(z'-y',y_n,t)h_n(y)dydz')\\
&\quad+4\sum_{k=1}^{n-1}D_{x_k}^2\int_{\Rn}N(x'-z',x_n)\int_{\R_+}\Gamma(z'-y',y_n,t)h_n(y)dydz'\\
&=D_t\pi_{00}+\sum_{k\neq n}D_{x_k}\pi_{2k}.
\end{align*}
From \eqref{Poisson},   Lemma \ref{trace} and Lemma \ref{gauss.equiv}, for  $\al>\frac{1}{p}$ we have
\begin{align}
\notag  \| \pi_{2k} \|_{ L^q(0,\infty; \dot B^{\al}_{pq}(\R_+))}
 \leq c
            \| \Ga_t * h_n|_{x_n =0}  \|_{ L^q(0,\infty; \dot B^{\al -\frac1p}_{pq} ({\mathbb R}^{n-1})}\\
  \leq  c  \|   \Ga_t * h_n \|_{ L^q(0,\infty; \dot B^{\al}_{pq}(\R_+))}
  \leq c \| h_n\|_{\dot{B}^{\al-\frac{2}q}_{pq,0}(\R_+) }
\end{align}
and
\begin{align}
\notag  \| \pi_{00} \|_{ L^q(0,\infty; \dot H^{\al+1}_p(\R_+))}
 \leq c
            \|  \Ga_t * h_n|_{x_n =0}  \|_{ L^q(0,\infty; \dot B^{\al -\frac1p}_p ({\mathbb R}^{n-1})}\\
  \leq  c  \|   \Ga_t * h_n \|_{ L^q(0,\infty; \dot H^{\al}_p(\R_+))}
  \leq c \| h_n\|_{\dot{B}^{\al-\frac{2}q}_{pq,0}(\R_+) }.
\end{align}
This completes  the proof of Lemma \ref{lemma-1-2-zero}.

\section{Proof of Lemma \ref{lemma-1-2}}
\label{appendix.lemma-1-2}

Recalling the formulae
\eqref{expression-p} with \eqref{formulas-p}, we split $\Pi_0$ in two terms,
i.e. $\Pi_0(x,t) = \Pi_1 (x,t) + \Pi_2(x,t)$, where
\begin{align*}
\Pi_1(x,t) =4 \sum_{1 \leq j \leq n-1}\int_{{\mathbb R}^{n-1}}
\frac{\pa^2 N(x'-z',x_n)}{\pa x_j \pa x_n} \int_0^t \int_{{\mathbb
R}^n_+} \Ga(z' -y', y_n, t -\tau) ({\mathbb P} f)_j(y, \tau) dyd\tau dz',
\end{align*}
\begin{align*} 
\Pi_2(x,t) =4 \sum_{1 \leq j \leq n-1}\int_{{\mathbb R}^{n-1}}
\frac{\pa N(x'-z',x_n)}{ \pa x_j} \int_0^t \int_{{\mathbb R}^n_+}
\frac{\pa \Ga(z' -y', y_n, t-\tau)}{\pa y_n}({\mathbb P} f)_j(y,\tau) dyd\tau dz'.
\end{align*}

Note that $\Pi_1$ is represented by
\begin{align*}
\Pi_1(x,t) = 4 \sum_{1 \leq j \leq n-1}  D_{x_j}  D_{x_n} N (\Ga *
({\mathbb P} f)_j|_{x_n =0})(x,t)= 4 \sum_{1 \leq j \leq n-1} D_{x_j} \Pi_{1j}(x,t),
\end{align*}
where $\Pi_{1j}(x,t) = 4  D_{x_n} N (\Ga *
({\mathbb P} f)_j|_{x_n =0}).$
From \eqref{Poisson},   Lemma \ref{trace} and Lemma \ref{0929-1}, for  $\al>\frac{1}{p}$ we have
\begin{align}
\notag  \| \Pi_{1j} \|_{ L^q(0,\infty; \dot H^{\al}_p(\R_+))}
 \leq c
            \|  \Ga * ({\mathbb P} f)_j|_{x_n =0}  \|_{ L^q(0,\infty; \dot B^{\al -\frac1p}_p ({\mathbb R}^{n-1})}\\
  \leq  c  \|   \Ga * ({\mathbb P} f)_j \|_{ L^q(0,\infty; \dot H^{\al}_p(\R_+))}
  \leq c \| {\mathcal F}\|_{L^{q_1} (0,\infty;  \dot H^{\be }_{p_1}(\R_+)) }.
\end{align}

By integrating by parts, by the identity $D_{y_n}({\mathbb P}f)_n=-\sum_{j=1}^{n-1}D_{y_j}({\mathbb P}f)_j$ and by the identity $D_{y_n}^2 \Gamma_t*({\mathbb P}f)_n=D_t\Gamma_t*({\mathbb P}f)_n-\Delta_{y'}\Gamma_t*({\mathbb P}f)_n$, $\Pi_2$ can be rewritten by
\begin{align*}
\Pi_2&=-4\int_{\Rn}N(x'-z',x_n)\int^t_0\int_{\R_+}D_{y_n}\Gamma(z'-y',y_n,t-s)D_{y_n}({\mathbb P}f)_n(y,s)dydsdz'\\
&=D_t\Big(4\int_{\Rn}N(x'-z',x_n)\int_{\R_+}\Gamma(z'-y',y_n,t)h_n(y)dydz')\\
&\quad+4\sum_{k=1}^{n-1}D_{x_k}^2\int_{\Rn}N(x'-z',x_n)\int^t_0\int_{\R_+}\Gamma(z'-y',y_n,t-s)({\mathbb P}f)_n(y,s)dydsdz'\\
&=D_t\Pi_{00}+\sum_{k\neq n}D_{x_k}\Pi_{2k}.
\end{align*}
From \eqref{Poisson},   Lemma \ref{trace} and  Lemma \ref{0929-1}, for  $\al>\frac{1}{p}$ we have
\begin{align}
\notag  \| \Pi_{2k} \|_{ L^q(0,\infty; \dot H^{\al}_p(\R_+))}
 \leq c
            \| \Ga *({\mathbb P}f)_n|_{x_n =0}  \|_{ L^q(0,\infty; \dot B^{\al -\frac1p}_p ({\mathbb R}^{n-1})}\\
  \leq  c  \|   \Ga *({\mathbb P}f)_n \|_{ L^q(0,\infty; \dot H^{\al}_p(\R_+))}
  \leq c \| {\mathcal F}\|_{L^{q_1}(0,\infty; \dot H^{\be}_{p_1}(\R_+) }
\end{align}
and
\begin{align}
\notag  \| \Pi_{00} \|_{ L^q(0,\infty; \dot H^{\al+1}_p(\R_+))}
 \leq c
            \|  \Ga *({\mathbb P}f)_n|_{x_n =0}  \|_{ L^q(0,\infty; \dot B^{\al -\frac1p}_p ({\mathbb R}^{n-1})}\\
  \leq  c  \|   \Ga * ({\mathbb P}f)_n \|_{ L^q(0,\infty; \dot H^{\al}_p(\R_+))}
  \leq c \| {\mathcal F}\|_{L^{q_1}(0,\infty; \dot H^{\be}_{p_1}(\R_+) }.
\end{align}
This completes  the proof of Lemma \ref{lemma-1-2}.

\end{document}